\documentclass[12pt]{article}

\usepackage{graphicx}
\usepackage{geometry}
\usepackage{amsfonts}
\usepackage{amsmath}
\usepackage{enumitem}
\usepackage{amsthm}
\usepackage{hyperref}
\usepackage{xcolor}
\usepackage[font=small,labelfont=bf, width=.7\textwidth]{caption}

\newtheorem{theorem}{Theorem}[section]
\newtheorem{lemma}[theorem]{Lemma}
\newtheorem{corollary}[theorem]{Corollary}

\newtheorem{claim}{Claim}
\newtheorem{conjecture}[theorem]{Conjecture}

\newenvironment{proofc}{{\noindent \textit{Proof of Claim.}}}{\hfill $\Box$\bigskip}

\newcommand{\jm}[1]{#1}
\newcommand{\am}[1]{#1}
\newcommand{\mdv}[1]{#1}

\begin{document}

\title{Non-monochromatic Triangles in a 2-Edge-Coloured Graph}

\author{
   Matt DeVos \thanks{\jm{Department of Mathematics, Simon Fraser University, Burnaby, B.C., Canada V5A 1S6, mdevos@sfu.ca.
     Supported in part by an NSERC Discovery Grant (Canada).}}
\and
   Jessica McDonald\thanks{\jm{Department of Mathematics and Statistics, Auburn University, Auburn, AL, USA 36849, mcdonald@auburn.edu.
   Supported in part by NSF grant DMS-1600551.}}
\and
   Amanda Montejano\thanks{\jm{UMDI Facultad de Ciencias, UNAM Juriquilla, Quer\'{e}taro, Mexico, amandamontejano@ciencias.unam.mx.
   \am{Supported  by CONACyT 219827, DGAPA: PASPA and PAPIIT IN14016.}}}
}

\date{}

\maketitle

\begin{abstract}
Let $G = (V,E)$ be a simple graph and let $\{R,B\}$ be a partition of $E$.  We prove that whenever $|E| + \min\{ |R|, |B| \} > { |V| \choose 2 }$, there exists a subgraph of $G$ isomorphic to $K_3$ which contains edges from both $R$ and $B$.  We conjecture a natural generalization to partitions with more blocks.
\end{abstract}

\section{Introduction}
Throughout all graphs are assumed to be simple.  A \emph{triangle} in a graph $G$ is a subgraph of $G$ isomorphic to a complete graph on 3 vertices.  One of the first theorems in extremal graph theory is the following best possible result.

\begin{theorem}[Mantel \jm{\cite{Ma}}]\label{Ma}
If $G = (V,E)$ satisfies $|E| > \frac{1}{4}|V|^2$, then $G$ contains a triangle.
\end{theorem}

In this paper we are interested in graphs equipped with edge colourings.  We define a $k$-\emph{edge-coloured} graph to be a graph $G = (V,E)$ equipped with a disjoint union $E = \sqcup_{i=1}^k E_i$ called a \emph{colouring}.   Every subgraph $H \subseteq G$ we will view as a $k$-edge-coloured graph equipped with the colouring $E(H) = \sqcup_{i=1}^k \left(E_i \cap E(H)\right)$.  We say that $G$ is \emph{monochromatic} if there exists $1 \le i \le k$ so that $E = E_i$, and otherwise it is \emph{non-monochromatic}.  If $|E_i| \le 1$ holds for every $1 \le i \le k$ we call $G$ \emph{rainbow}.

One of the central problems in Ramsey Theory is to determine for every $k$ and $t$ the smallest integer $n$ so that every $k$-edge-coloured $K_n$ contains a complete subgraph of order $t$ which is monochromatic.  On the flip side, problems \jm{in} Anti-Ramsey Theory are concerned with conditions forcing the existence of rainbow subgraphs. \jm{See \cite{CFS} and \cite{FMO} for recent surveys of Ramsey Theory and Anti-Ramsey theory, respectively.} Here it is worth noting that Gallai has given a precise description of all possible edge-colourings of a complete graph that do not have a rainbow triangle.   Our problem is another variation on these themes.  Here we are focused on conditions on an arbitrary graph forcing the existence of a non-monochromatic triangle.  We prove the following \mdv{theorem and immediate corollary:}

\begin{theorem}
\label{main1}
If $G = (V,E)$ is a 2-edge-coloured graph with colouring $E = E_1 \sqcup E_2$, and
$|E| + \min\{ |E_1|, |E_2| \} > {|V| \choose 2}$, then $G$ contains a non-monochromatic triangle.
\end{theorem}

\mdv{\begin{corollary} \label{maincor} If $G = (V,E)$ is a 2-edge-coloured graph with colouring $E= E_1 \sqcup E_2$ and
$|E_1|, |E_2| > \frac{1}{3}{ |V| \choose 2}$, then $G$ contains a non-monochromatic triangle.
\end{corollary}}

\mdv{Theorem \ref{main1}} is tight in the sense that there exist graphs missing the bound by one edge for which there is no non-monochromatic triangle (i.e. all triangles are monochromatic).  We have characterized these graphs and we introduce them next.  For this discussion, we need to extend the usual notion of isomorphism to coloured graphs:  If $G=(V,E)$ and $G' = (V',E')$ are $k$-edge-coloured graphs relative to the colourings $E = \sqcup_{i=1}^k E_i$ and $E' = \sqcup_{i=1}^k E_i'$, then we say that $G$ and $G'$ are \emph{isomorphic} if there is an isomorphism from $G$ to $G'$ that maps every $E_i$ to some $E_j'$.

For every positive integer $m$, define the 2-edge-coloured graph $H_m$ as follows.  We let $V(H_m) = \sqcup_{i=0}^2 X_i$ where $|X_i| = m$ for $0 \le i \le 2$.  We add to $H_m$ all edges between $X_0$ and $X_1 \cup X_2$ and add a clique on $X_1$ and a clique on $X_2$.  We equip $H_m$ with a 2-edge-colouring by defining $E_i$ to be all edges incident with a vertex in $X_i$ for $i=1,2$.  We let $H_m^+$ be the 2-edge-coloured graph defined exactly as above with the exception that $|X_0| = m+1$ \jm{(see Figure \ref{HmPlus})}.  As we prove, the only cases where $|E| + \min\{ \am{|E_1|, |E_2|} \} = { |V| \choose 2}$ for which no non-monochromatic triangle exists are:
\begin{enumerate}
\item $G$ is a monochromatic complete graph.
\item $G$ is a 4-cycle and $E_i$ is a perfect matching for $i=1,2$
\item $G$ is isomorphic to either $H_m$ or $H_m^+$ for some $m \ge 1$.
\end{enumerate}

\begin{figure}
  \centering
  \includegraphics[width=280pt]{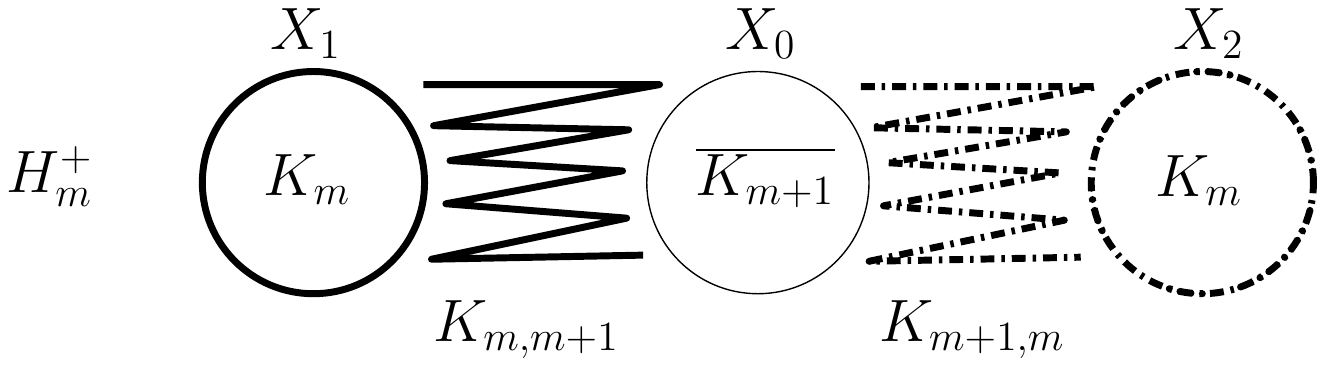}
  \caption{\jm{The graph $H_m^+$, with two colour classes indicated by the bold solid lines and bold dashed-dotted lines. Obtain $H_m$ from $H_m^+$ by deleting a vertex in $X_0$.}}
  \label{HmPlus}
\end{figure}

\mdv{Note that apart from monochromatic complete graphs, these examples satisfy $|E_1| = |E_2| = \frac{1}{3} {|V| \choose 2}$ showing that Corollary \ref{maincor} is also tight. }

%\jm{Observe that aside from monochromatic complete graphs, all these extremal examples satisfy $|E_1|=|E_2|$. In this special case, $|E| + \min\{ |E_1|, |E_2| \}= \tfrac{3}{2}|E|$, and hence Theorem \ref{main1} implies the following.}
%
%\jm{\begin{corollary} If $G = (V,E)$ is a 2-edge-coloured graph with colour classes of equal size, and
%$|E|> \tfrac{2}{3}{|V| \choose 2}$, then $G$ contains a non-monochromatic triangle.
%\end{corollary}}

We conjecture the following many colour variation\mdv{s} to Theorem \ref{main1} \mdv{and Corollary~\ref{maincor}:}

\am{\begin{conjecture}\label{cjt0}
If $G = (V,E)$ is a $k$-edge-coloured graph with colouring $E = \sqcup_{i=1}^k E_i$ and
$2|E|-\max\{|E_i|\,:\,1\leq i\leq k\}> \frac{1}{2} |V|^2$ then $G$ contains a non-monochromatic triangle.
\end{conjecture}}

\begin{conjecture}\label{cjt}
If $G = (V,E)$ is a $k$-edge-coloured graph with colouring $E = \sqcup_{i=1}^k E_i$ and
$|E_i| \mdv{>} \frac{1}{4k-2}|V|^2$ holds for $1 \le i \le k$, then $G$ contains a non-monochromatic triangle.
\end{conjecture}

If true, \mdv{these conjectures are} \mdv{essentially} \am{best possible} thanks to a generalization of the graph $H_m$.  For positive integers $k \ge 1$ and $m \ge 1$ define the $k$-edge-coloured graph $H_m^k$ as follows.  The vertex set of $H_m^k$ is the disjoint union of sets $X_0, \ldots, X_k$ where $|X_0| = (k-1)m$ and $|X_i| = m$ for $1 \le i \le m$.  Add to $H_m$ all edges between $X_0$ and $\cup_{i=1}^k X_i$ and add a clique on $X_i$ for $1 \le i \le k$.  Equip $H_m^k$ with a $k$-edge-colouring by defining $E_i$ to be all edges incident with a vertex of $X_i$ for $1 \le i \le k$.  Now \mdv{$H_m^k$ does not have a non-monochromatic triangle, but $|V(H_m^k)| = (2k-1)m$ and every $E_i$ satisfies
\[ |E_i| = \mbox{${m \choose 2}$} + (k-1)m^2 \approx (k - \tfrac{1}{2})m^2 = \tfrac{1}{4k-2} |V(H_m^k)|^2. \]
}

%
%$|V(H_m^k)| = (2k-1)m$ and every $E_i$ has $|E_i| = {m \choose 2} + (k-1)m^2$  \am{(see Figure \ref{Hkm}). Then,
%\begin{align}
%2|E|-\max\{|E_i|\,:\,1\leq i\leq k\}&=(2k-1)|E_i|\nonumber\\
% &=(2k-1)\left[{m \choose 2} + (k-1)m^2\right]\nonumber\\
% &=\frac{(2k-1)m}{2}\left[(m-1) + 2(k-1)m\right]\nonumber\\
% &= \frac{(2k-1)m}{2}\left[(2k-1)m-1\right]={|V(H_m^k)| \choose 2}.\nonumber
%\end{align}
%}  However the graph $H_m^k$ does not have a non-monochromatic triangle.
%
%\am{The following is a  weaker form of Conjecture \ref{cjt0}  which provides the flavour of Theorem \ref{Ma}.}
%

\begin{figure}
  \centering
  \includegraphics[width=190pt]{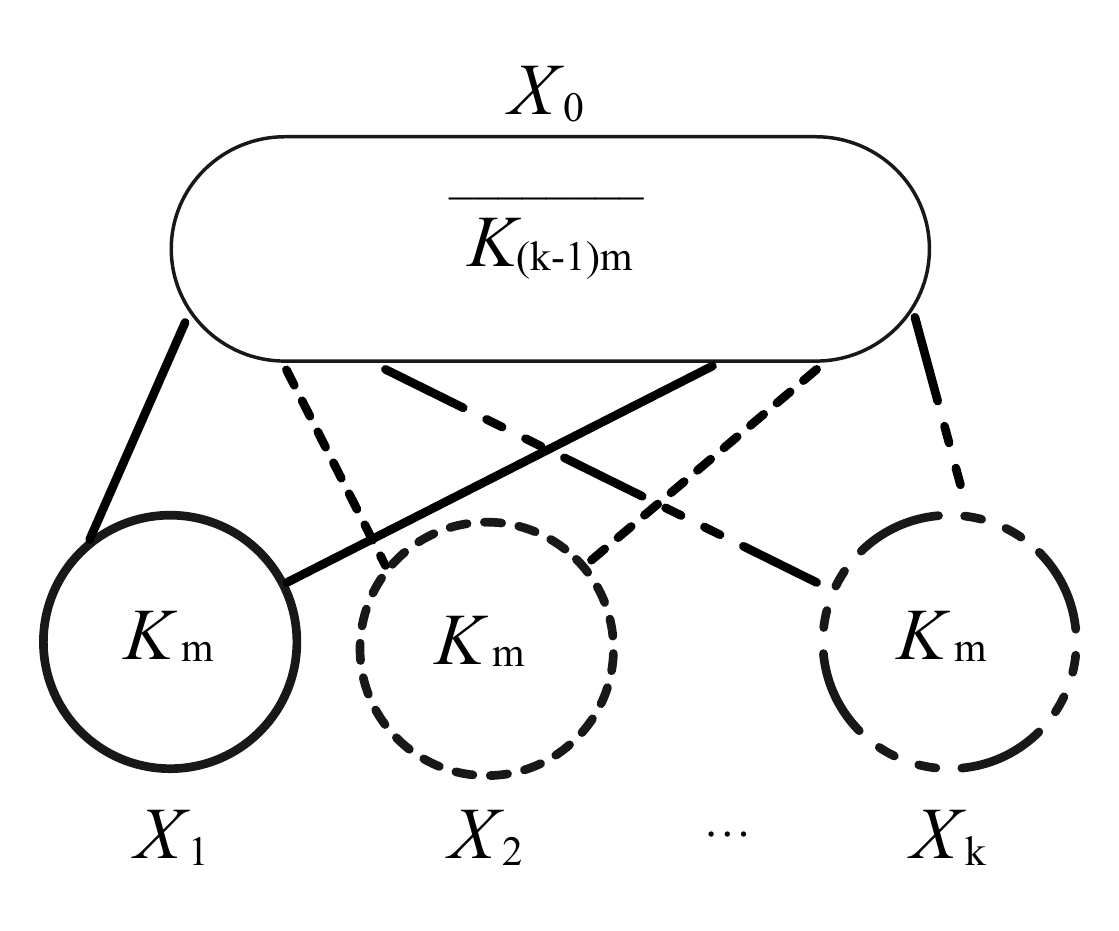}
  \caption{\am{The graph $H_m^k$ with $k$ color clases. It indicates that Conjecture \ref{cjt0} is best possible.}}
  \label{Hkm}
\end{figure}

\mdv{
One might expect a stronger form Conjecture \ref{cjt0} to hold under the weaker assumption $2|E|-\max\{|E_i|\,:\,1\leq i\leq k\}> {|V| \choose 2}$ (indeed we originally did!).  However, Sophie Spirkl noted that a rainbow 4-cycle is a counterexample to this.   }

\section{The Proof}

For a graph $G=(V,E)$ we let $\alpha(G)$ denote the size of the largest independent set.  We define the \emph{density} of $G$ to be:
\[ d(G) = \frac{ |E| }{{ |V| \choose 2 }}. \]
We have introduced the concept of density in part because we find it easiest to think in these terms.  However, it is used somewhat sparingly in the proof because it is usually easier to write out our bounds in terms of edges.  The proof of our main theorem calls on the following straightforward lemma.

\begin{lemma}
\label{thelemma}
If $G$ is a graph with $|V(G)| \ge 2$ and density at least $\frac{2}{3}$, there is a \jm{component $G'$ of $G$ with $|V(G')|>\tfrac{2}{3}|V(G)|$}, and furthermore,
\[ \alpha(G') +1 \le 2 |V(G')| - |V(G)|. \]
\end{lemma}

\begin{proof}
Since $G$ must have a vertex with at least $\frac{2}{3}( |V(G)| - 1 )$ neighbours, it follows that there is a component $G'$ of $G$ containing at least $\frac{2}{3} |V(G)|$ vertices.  Define $a = \alpha(G')$, let $b = |V(G')| - a$ and  $c = |V(G)| - |V(G')|$, and note that $|V(G')| \ge \frac{2}{3} |V(G)|$ implies $a + b \ge 2c$.  Also note that $b > 0$ since otherwise the component $G'$ would be forced to be an isolated vertex.  Assume (for a contradiction) that $\alpha(G') \ge 2 |V(G')| - |V(G)| $, and note that this implies $c \ge b$.  Together with the bound $a + b \ge 2c$ we then have $a \ge c \ge b$.  But this gives the contradiction
\[ |E(G)| \le  \mbox{${b \choose 2}$} + \mbox{${c \choose 2}$} + ab  < \tfrac{2}{3} \Big( \mbox{${a \choose 2}$} +  \mbox{${b \choose 2}$} +  \mbox{${c \choose 2}$}   +  ab + ac + bc \Big) = \tfrac{2}{3} \mbox{${|V(G)| \choose 2}$}. \qedhere \]
\end{proof}

We need a little additional notation before proving the main result.  Let $G= (V,E)$ be a 2-edge-coloured graph with colouring $E = R \sqcup B$.  If $X \subseteq V$ we let $G[X]$ denote the 2-edge-coloured subgraph of $G$ induced by $X$.  We let $E(X)$ denote the set of edges with both ends in $X$, and define:
\[ e(X) = |E(X)| \qquad	e_B(X) = |E(X) \cap B| \qquad e_R(X) = |E(X) \cap R|.	\]
If $Y \subseteq V$ is disjoint from $X$, we let $E(X,Y)$ denote the set of edges with one end in $X$ and one end in $Y$.  Define:
\[ e(X,Y) = |E(X,Y)|	\qquad	e_B(X,Y) = |E(X,Y) \cap B|	\qquad	e_R(X,Y) = |E(X,Y) \cap R|. \]

Next we prove our main theorem, restated to allow for $|E| + \min\{ |R|, |B| \} \ge { |V| \choose 2 }$.

\begin{theorem}
\label{main2}
Let $G = (V,E)$ be a 2-edge-coloured graph with colouring $E = R \sqcup B$.  If
$|E| + \min\{ |R|, |B| \} \ge { |V| \choose 2 }$, then one of the following holds:
\begin{enumerate}
\item $G$ has a non-monochromatic triangle.
\item $G$ is a monochromatic clique.
\item $G$ is a 4-cycle and $R$, $B$ are perfect matchings.
\item $G$ is isomorphic to either \am{$H_m$ or $H_m^+$ for some $m \ge 1$}.
\end{enumerate}
\end{theorem}

\begin{proof}
Suppose (for a contradiction) that the theorem is false, and let \am{$G=(V,E)$} be a counterexample with \am{$|V|$} minimum.  Note that this implies $G$ is connected.  \mdv{We call edges in $R$ \emph{red} and edges in $B$ \emph{blue}, and for a vertex $v \in V$ an adjacent vertex is called a \emph{red} (\emph{blue}) neighbour if it is joined to $v$ by a red (blue) edge.}  We say that a subset $X \subseteq V$ with $|X| = 3$ is a \emph{seagull} if $X$ is the vertex set of a non-monochromatic two edge path $P$.  Note that since $G$ does not have a non-monochromatic triangle, the subgraph $P$ must be induced.  We say that a subset $\am{X} \subseteq V$ with $|\am{X}| = 4$ is an \emph{alternating square} if \am{$X$} is the vertex set of a properly coloured 4-cycle $C$ (i.e. every $x \in X$ is incident with one edge in $R \cap E(C)$ and one in $B \cap E(C)$).  Again in this case the subgraph $C$ must be induced.  We begin with a couple of straightforward claims about the behaviour of $G$.

\bigskip

\begin{claim}\label{d23}  $d(G) \ge \frac{2}{3}$ and if $d(G) = \frac{2}{3}$, then $|R| = |B|$.
\end{claim}

\begin{proofc}
This follows from $|E| + \tfrac{1}{2} |E| \ge |E| + \min\{ |R|, |B| \} \ge { |V| \choose 2}$.
\end{proofc}

\begin{claim}\label{seagull} If $\jm{X}$ is a seagull \am{in $G$} and $v \in V \setminus \jm{X}$, then $e(v, \jm{X} ) \le 2$
\end{claim}

\begin{proofc}
\jm{If $v$ has an edge to each of the three vertices in \jm{$X$} then, regardless of their colours, we get a non-monochromatic triangle.}
\end{proofc}

\begin{claim}\label{C4} $G$ does not have an alternating square
\end{claim}

\begin{proofc}
Suppose (for a contradiction) that $X \subseteq V$ is an alternating square.  Note that $X \neq V$ since in this case $G$ satisfies the third outcome of the theorem.  Define $Y = V \setminus X$ and note that by \am{Claim \ref{seagull}}, every $y \in Y$ must satisfy $e(y,X) \le 2$.  In particular, this implies
\begin{equation}\label{XY}
e(X,Y) \le \tfrac{1}{2} |X| |Y|.
\end{equation}
It follows from \am{(\ref{XY})} that
\begin{eqnarray}\label{YVR}
e(X) + e(X,Y) + \max\big\{&&\hspace*{-.4in} e_B(X) + e_B(X,Y), e_R(X) + e_R(X,Y) \big\}\nonumber\\
&\leq& \jm{4+ \tfrac{1}{2} |X| |Y| + (2+\tfrac{1}{2} |X| |Y|) = } \mbox{${ |X| \choose 2 }$} + |X| |Y|
\end{eqnarray}
\jm{We claim that $e(Y) + \min\{ e_R(Y), e_B(Y) \} \geq { |Y| \choose 2 }$. If not, then suppose, without loss of generality, that $e(Y) + e_R(Y) \am{<} { |Y| \choose 2 }$.}  \mdv{Equation (\ref{YVR}) implies that 
$e(X) + e(X,Y) + e_R(X) + e_R(X,Y) \le { |X| \choose 2 } + |X| |Y|$.  Summing this and $e(Y) + e_R(Y) \am{<} { |Y| \choose 2 }$ give us $|E| + |R| < {|V| \choose 2}$ which contradicts our primary assumption.}
%\am{and so,
%\begin{equation}\label{XXY}
%{ |Y| \choose 2 }+e(X) + e(X,Y) + \max\big\{ e_B(X) + e_B(X,Y), e_R(X) + e_R(X,Y) \big\}\leq { |V| \choose 2 }.
%\end{equation}
%}
%\jm{We claim that $e(Y) + \min\{ e_R(Y), e_B(Y) \} \geq { |Y| \choose 2 }$. If not, then suppose, without loss of generality, that $e(Y) + e_R(Y) \am{<} { |Y| \choose 2 }$. Coupled with (\ref{XXY}), this tells us that}
%\am{
%\begin{equation}
%|E|+ e_R(Y) + \max\big\{ e_B(X) + e_B(X,Y), e_R(X) + e_R(X,Y) \big\}< { |V| \choose 2 },\nonumber
%\end{equation}
%from which it follows that
%\begin{align}
%|E|+|R|&=|E|+e_R(Y) + e_R(X,Y)+e_R(X)\nonumber\\
% &\leq |E|+e_R(Y)+\max\big\{ e_B(X) + e_B(X,Y), e_R(X) + e_R(X,Y) \big\}< { |V| \choose 2 }\nonumber
% \end{align}
%a contradiction to} \jm{our primary assumption.
\jm{Hence $e(Y) + \min\{ e_R(Y), e_B(Y) \} \am{\geq} { |Y| \choose 2 }$, and by the minimality of $G$, the theorem holds for $G[Y]$.}
\am{Thus,} $G[Y]$ must be either a monochromatic clique, an alternating square, or $G[Y]$ is isomorphic to one of \am{$H_m$ or $H_m^+$ for some $m \ge 1$}.  In all but the first case, the density of $G[Y]$ is equal to $\frac{2}{3}$, \jm{and since the density of $G[X]$ is also $\tfrac{2}{3}$, inequality (\ref{XY}) implies that $d(G)<\tfrac{2}{3}$, contradicting Claim \ref{d23}.}  In the remaining case, $G[Y]$ is a clique and we may assume, \jm{without loss of generality, that} all edges of $G[Y]$ are in $R$.  Since $|E| + |B| \ge { |V| \choose 2 }$ %\jm{and $G[X]$ and $G[Y]$ satisfy this with equality (i.e. $e(X)+e_B(X)={|X| \choose 2}$ and $e(Y)+e_B(Y)={|Y| \choose 2}$), it must be that $e(X,Y)+e_B(X,Y)\geq |X||Y|$. However in light of (\ref{XY}), this means that}
it must be that $e(X,Y) = e_B(X,Y) = \frac{1}{2} |X| |Y|$ and every $y \in Y$ is incident with exactly two vertices in $X$.  If $|Y| \ge 3$, then there are two vertices in $Y$ with a common neighbour in $X$ and this triple forms a non-monochromatic triangle---a contradiction.  Otherwise, $1 \le |Y| \le 2$ and $e_R(X,Y) = 0$ give us a contradiction to the $|E| + |R| < { |V| \choose 2}$.
% Otherwise, \jm{$|Y|\in\{1, 2\}$, and in both cases $e_R(X,Y)=0$ means that the resulting $G$ contradicts $|E| + |R| \geq { |V| \choose 2}$. In particular, when $|Y|=1$ we have $|R|=2$ but there are three pairs of non-adjacent vertices in $G$, and when $|Y|=2$ then $|R|=3$ but there are four pairs of non-adjacent vertices in $G$}.
\end{proofc}

\begin{claim}\label{S1S2}
Let $S_1 , S_2$ be disjoint seagulls \am{in $G$. Let } $r = e_R(S_1,S_2)$ and $b = e_B(S_1,S_2)$, then
\begin{itemize}
\item $r + b \le 6$,
\item if $r+b = 6$ then $G[S_1 \cup S_2]$ is isomorphic to $H_2$, and
\item $r + b + \max \{ r, b \} \le 9$.
\end{itemize}
\end{claim}

\begin{proofc}
\jm{By Claim \ref{seagull}, there are at most 2 edges between a seagull and any other vertex, so in particular there are at most 6 edges between $S_1$ and $S_2$.} \am{If  $r+b = 6$ then each vertex in $S_1$  has exactly 2 neighbours in $S_2$.} \jm{Label the vertices of $S_1, S_2$ as $u, v, w$, $x, y, z$, respectively, with the edges $uv$ and $xy$ being in $R$, and the edges $vw$ and $yz$ being in $B$.} \am{Observe that  a blue  neighbor of $v$ cannot be a neighbor of $u$, and a red  neighbor of $v$  cannot be a neighbor of $w$. Hence, if each of  $u, v$ and $w$ has exactly 2 neighbours in $S_2$, it must be that $v$ has exactly one red neighbor  and one blue neighbor in $S_2$. Moreover, the red neighbour of $v$ in $S_2$ must be $x$ and the blue neighbour of $v$ in $S_2$ must be $z$ (otherwise, $v,x,y,z$ would induce alternating square contradicting Claim \ref{C4}).  This force  $G[S_1 \cup S_2]$ to be $H_2$ (see Figure \ref{TwoSeagulls3}). Note that the third instance of  the claim is satisfied in this case. It remains to prove the third instance of  the claim when $r+b\leq 5$. For that, it suffices to show that $\max\{r, b\} \le 4$. Suppose (for a contradiction, and without loss of generality) that $b=5$. Then, there are at least three $B$-edges between $\{u,v\}$ and $S_2$. Avoiding the existence of non-monochromatic triangles, the two possibilities for this to happen are depicted in Figure \ref{SeagullEdgeCases2}. In either case, we have created an alternating square,  a contradiction to Claim \ref{C4}.}
\end{proofc}

\begin{figure}
  \centering
  \includegraphics[width=350pt]{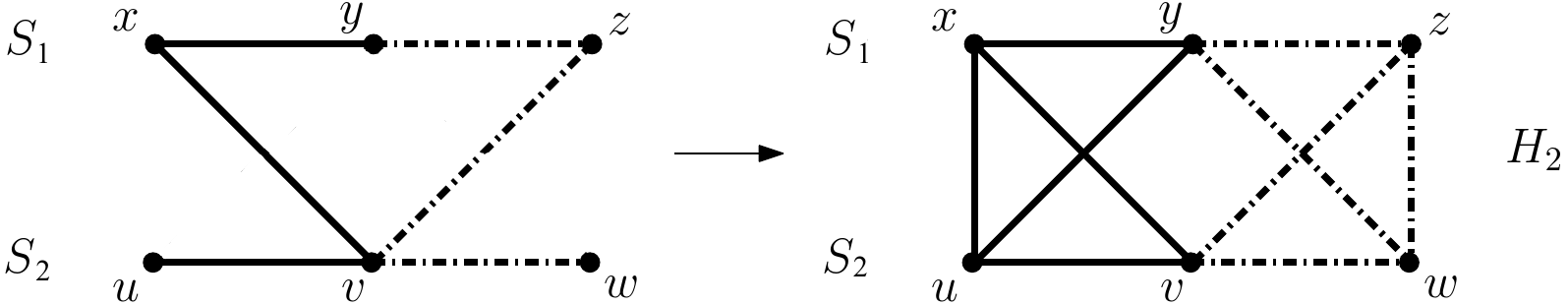}
  \caption{\jm{Forcing a copy of $H_2$ in the proof of Claim \ref{S1S2}. Here and elsewhere in the paper, solid bold lines indicate colour class $R$, and dashed-dotted bold lines indicate colour class $B$.}}
  \label{TwoSeagulls3}
\end{figure}

\begin{figure}
  \centering
  \includegraphics[width=300pt]{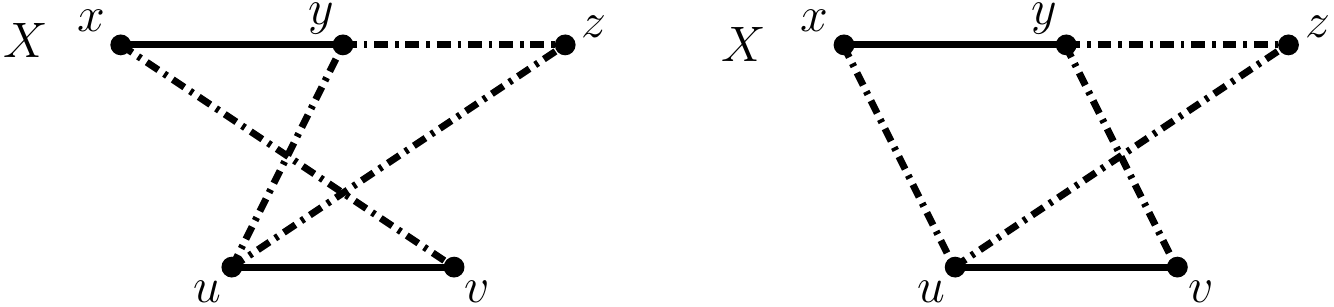}
  \caption{\jm{The two  cases to consider in Claim \ref{S1S2} \am{when we assume that $b=5$}.}}
  \label{SeagullEdgeCases2}
\end{figure}

Choose a maximum size list of vertex disjoint seagulls $S_1, \ldots, S_m$ \am{in $G$} and define the sets $S = \cup_{i=1}^m S_i$ and $T = V \setminus S$. \jm{By Claim \ref{seagull}, each vertex in $T$ is adjacent to at most two of the three vertices in each seagull $S_i$.} \am{Hence,}

\begin{equation}\label{ST}
e(S,T) \le \tfrac{2}{3} |S| |T|.
\end{equation}

\begin{claim}\label{S} \jm{The set $S$ satisfies
\begin{equation}\label{Seq}
\tfrac{3}{2} e(S) \le e(S) + \max\big\{ e_R(S), e_B(S) \big\} \le \tbinom{|S|}{2}.
\end{equation}
Moreover, if \am{$e(S) =\frac{2}{3} { |S| \choose 2 }$}, then $G[S]$ is isomorphic to $H_m$.}
\end{claim}

\begin{proofc} \jm{Since $e(S) + \tfrac{1}{2} e(S) \leq e(S) + \max\big\{ e_R(S), e_B(S) \big\}$, the first inequality in (\ref{Seq}) is immediate.} \am{The second inequality in (\ref{Seq}) follows from the third instance of  Claim \ref{S1S2}, which states that} \jm{for any distinct $S_i, S_j$ we have
\begin{equation}\label{s1s2}
e(S_i, S_j) + \max\{e_R(S_i, S_j), e_B (S_i, S_j)\}\leq 9=|S_i|\cdot|S_j|,
\end{equation}
and from the fact that} \jm{for every $S_i$, $G[S_i]$ has exactly one edge in $R$ and one edge in $B$, and $2+1={3\choose 2}$.} \am{Finally, if $e(S) = \frac{2}{3}{ |S| \choose 2 }$ then $e(S_i,S_j)=6$ for any $i\neq j$. Thus, by the second instance of Claim \ref{S1S2}, $G[S_i \cup S_j]$ is isomorphic to $H_2$  for any distinct $S_i, S_j$. This gives $G[S]$ isomorphic to $H_m$.}
\end{proofc}

\am{For a given $x\in V$ we will use $N_B(x)$, and $N_R(x)$, to denote the set of blue neighbours of $x$, and the set of red neighbours of $x$, respectively.}

\begin{claim}\label{STvalues} \jm{$S\neq \emptyset$ and $|T|\geq 3$.}
\end{claim}

\begin{proofc} If $S = \emptyset$, then the assumption \am{that} $G$ is connected implies that one of $R$ \am{or} $B$ must be empty, and then the only way for $G$ to satisfy $|E| + \min\{|R|, |B| \} \ge { |V| \choose 2 }$ is for $G$ to be a monochromatic clique.  Therefore $m \ge 1$.

\am{Note that,  if $|T|\leq 2$ the only way for $G$ to satisfy Claim \ref{d23} is to have equalities in both (\ref{ST}) and (\ref{Seq}), that is
\begin{equation}\label{23}
e(S, T)=\tfrac{2}{3}|S||T| \mbox{ and } e(S)=\tfrac{2}{3}\tbinom{|S|}{2},
\end{equation}
and furthermore, if $|T|=2$, we must have $e(T)=1$. The second identity in (\ref{23})  implies that $G[S]$ is isomorphic to $H_m$ (by   Claim \ref{S}).  \mdv{So, we may write $V(G[S]) = X_0\sqcup X_1 \sqcup X_2$ as in Figure 1 where the edges incident with a vertex in $X_1$ are red and those incident with a vertex in $X_2$ are blue.}
\begin{itemize}
\item If $|T| = 0$ then $G=G[S]$ is isomorphic to $H_m$, a contradiction.
\item If $|T|=1$, say $T=\{v\}$. In this case we have $d(G)=\tfrac{2}{3}$, thus $|R|=|B|$ (by Claim \ref{d23}). Since $G[S]$ is isomorphic to $H_m$, and $e_R(H_m)=e_B(H_m)$ then we must have $e_R(v, S)=e_B(v, S)=m$. \mdv{First observe that $v$ cannot have a neighbour $x \in X_0$: If say $vx \in R$, then $N_B(v) \subset X_0 \setminus \{x\}$ is too small.  It follows from this and Claim \ref{C4} that $N_R(v)=X_1$ and $N_B(v)=X_2$, so $G$ is isomorphic to $H_m^+$, a contradiction.}
%
%
%
%Suppose (for a contradiction) that $vx\in R$ for some $x\in X_0 \sqcup X_2$; if $x\in X_2$ then
% $v$ has no more  neighbours in  $X_0\sqcup X_2$, so $N_B\sqcup (N_R(v)\setminus\{x\})\subset X_1$ which is imposible by size; if $x\in X_0$ then a blue neighbour of $v$ cannot be in $ X_1 \sqcup X_2$, so $N_B(v)\subseteq X_0\setminus \{x\}$ which is imposible by size. Hence,  $N_R(v)=X_1$ and, similarly,  $N_B(v)=X_2$. This tell us  that $G$ is isomorphic to $H_m^+$, a  contradiction.
\item If $|T|=2$,  say $T=\{u,v\}$, and suppose without loss of generality that $uv\in R$. Since $G[S]$ is isomorphic to $H_m$, and $e_R(H_m)=e_B(H_m)$ then, in order  to satisfy $|E|+\min\{|R|, |B|\}\geq \tbinom{|V|}{2}$, it must be that $e_R(S,T)=e_B(S,T)=2m$. Thus,  there is a seagull $S_i$ with $e_B(T,S_i)\geq 2$. Recall that $e(u, S_i)=e(v, S_i)=2$. If $|N_B(x)\cap S_i|= 2$ for some $x\in\{u,v\}$ then there is at most one edge between $y$ and $S_i$, where $y\in \{u,v\}\setminus \{x\}$. Hence, $|N_B(x)\cap S_i|=|N_R(x)\cap S_i|= 1$ for both $x\in\{u,v\}$ which is impossible without creating either an alternating square or a non-monochromatic triangle.
\end{itemize}}
\end{proofc}

Now we have $S \neq \emptyset$ and $|T| \ge 3$ and we will show that this leads to a contradiction.  Note that the graph $G[T]$ must have all components monochromatic (by the maximality of our collection of seagulls).  It follows from \mdv{Claim \ref{d23}, Equation (\ref{ST}), and Claim \ref{S}} that $G[T]$ has density at least $\frac{2}{3}$.  \jm{Hence by Lemma \ref{thelemma}, $G[T]$ has a component $H$ with $|V(H)|>\tfrac{2}{3}|T|$.}  We shall assume (without loss) that $E(H) \subseteq R$.  Next we establish a key claim concerning blue edges between our seagulls and $V(H)$.

\begin{claim}\label{SH} For every $1 \le i \le m$ we have $e_B(S_i, V(H)) \le \alpha(H) + 1$.
\end{claim}

\begin{proofc}
To prove this claim, let $S_i = \{x,y,z\}$ be a seagull and assume that $xy \in R$ and $yz \in B$.
\am{Since} $G$ does not have an alternating cycle \am{(by Claim \ref{C4})} or a non-monochromatic triangle \am{(by assumption)} it follows that $N_B(x) \cap V(H)$ and $N_B(y) \cap V(H)$ are disjoint and moreover \am{$(N_B(x) \cup N_B(y)) \cap V(H)$} is independent (\jm{see the left image in Figure \ref{SeagullH2}}).  Similarly, $N_B(z) \cap V(H)$ is independent.  If $|(N_B(x) \cup N_B(y) ) \cap V(H)| \le 1$ the desired bound follows immediately, and we are similarly done if $|N_B(\mdv{z}) \cap V(H)| \le 1$.  So we may assume both of these sets have size at least two.  Choose $w \in N_B(z) \cap V(H)$ and (since $H$ is nontrivial) choose an edge $ww' \in E(H)$.  At most one of $w, w'$ can appear in $(N_B(x) \cup N_B(y) ) \cap V(H)$ since this set is independent.  Therefore, we may choose a vertex $u \in (N_B(x) \cup N_B(y) ) \cap V(H)$ with $u \neq w, w'$.  Now we have arrived at a contradiction as $\{z,w,w'\}$ and $\{x,y,u\}$ are both seagulls \jm{(see the right image in Figure \ref{SeagullH2})}.
\end{proofc}

\begin{figure}
  \centering
  \includegraphics[width=310pt]{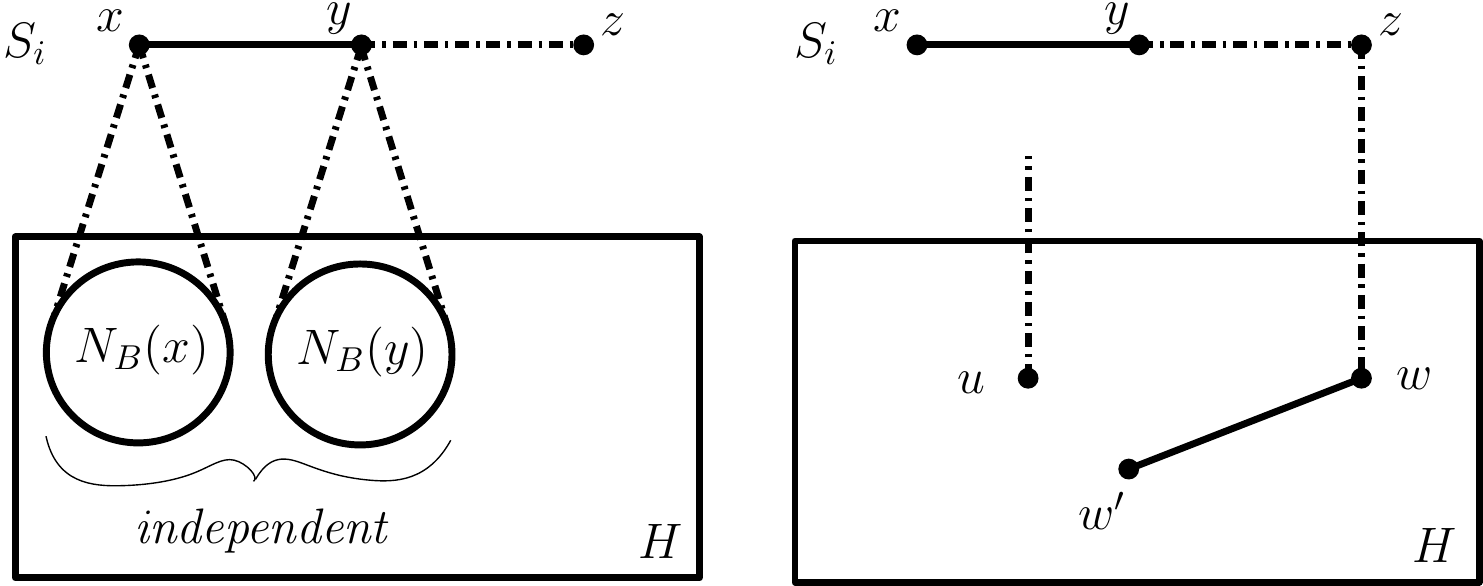}
  \caption{\jm{The situation under consideration in Claim \ref{SH}.}}
  \label{SeagullH2}
\end{figure}

\jm{In what follows, we use $|V(H)|=h$, $|T|=t$, and $|S|=s$.} For every seagull $S_i$, Claim 2 implies that $e(S_i,T) \le 2t$ and $e_B(S_i, T \setminus V(H)) \le e(S_i, T \setminus V(H)) \le 2(t-h)$.  Using these inequalities, together with Claim \ref{SH} and then Lemma \ref{thelemma}, we have
\begin{align}
e(S_i, T) + e_B(S_i,T)
	&=	e(S_i, T) + e_B(S_i, V(H)) + e_B(S_i, T \setminus V(H) )	\nonumber \\
	&\le	2t + \alpha(H) + 1 + 2(t-h)	\label{pre3t}
\\
	&\le	3t. \label{3t}
\end{align}
Summing this over all of our seagulls gives us $e(S,T) + e_B(S,T) \le st$.  Claim \ref{S} implies that $e(S) + e_B(S) \le { s \choose 2 }$.  Combining these with the assumption $|E| + |B| \ge { |V| \choose 2 }$ we deduce that
\begin{equation}
\label{Tdensity}
e(T) + e_B(T) \ge \mbox{${ t \choose 2}$}.
\end{equation}
\jm{Since $e(V(H), T\setminus V(H))=0$, we know that  $\tbinom{t}{2}-e(T) \geq h(t-h)$, and hence (\ref{Tdensity}) says that $e_B(T)\geq h(t-h)$. On the other hand, if $h\neq t$, then the fact that $h\geq t-h$ (by Lemma \ref{thelemma}) implies that $e_B(T)\leq \tbinom{t-h}{2} <(t-h)^2\leq h(t-h)$.} Therefore, we must have $V(H) = T$ and $e_B(T) = 0$.  Now the only way for inequality (\ref{Tdensity}) to hold is for $H$ to be a clique; \jm{in particular this means $\alpha(H)=1$}. Since $t \ge 3$ \jm{(by Claim \ref{STvalues}) we can return to (\ref{pre3t}) and improve upon (\ref{3t})}:
\[ e(S_i, T) + e_B(S_i, T) \le 2t + 2 < 3t. \]
\jm{Summing over all of our seagulls gives us $e(S,T) + e_B(S,T) < st$}.  Combining this with the bound $e(S) + e_B(S) \le {|S| \choose 2}$ (from  \jm{Claim (\ref{S})}) gives us the contradiction
$$|E| + |B| = e(S) + e_B(S) + e(S,T) + e_B(S,T) + \tbinom{t}{2} < \tbinom{ |V|}{2 }$$
and this completes the proof.
\end{proof}

\jm{\bibliographystyle{amsplain} \bibliography{non-mono-triangleBIB}}

\end{document}